\documentclass[11pt,reqno]{amsart}

\usepackage{amsmath, amsfonts, amsthm, amssymb}

\newtheorem{Theorem}{Theorem}[section]
\newtheorem{Lemma}{Lemma}[section]
\newtheorem{Proposition}{Proposition}[section]

\theoremstyle{definition}

\theoremstyle{remark}
\newtheorem{Remark}{Remark}[section]

\numberwithin{equation}{section}

\renewcommand{\u}{{\bf u}}

\newcommand{\R}{{\mathbb R}}
\newcommand{\Dv}{{\rm div}}

\newcommand{\m}{{\bf m}}

\def\w{w}

\def\f{\frac}
\renewcommand{\O}{\Omega}

\def\D{\Delta }

\def\hf1{^\f{1}{1-\xi^2}}

\def\be{\begin{equation}}
\def\en{\end{equation}}
\def\bs{\begin{split}}
\def\es{\end{split}}

\author{Alexis F. Vasseur}
\address{Department of Mathematics,
The University of Texas at Austin.}
\email{vasseur@math.utexas.edu}
\author{Cheng Yu}
\address{Department of Mathematics,
The University of Texas at Austin.}
\email{yucheng@math.utexas.edu}
\title[Global solutions to an approximated system]
{Global Weak Solutions to Compressible Quantum Navier-Stokes Equations with Damping}

\date{\today}

\begin{document}
\begin{abstract}The global-in-time existence of weak solutions to  the barotropic compressible quantum Navier-Stokes equations  with damping is proved for large data in three dimensional space. The model consists of the compressible Navier-Stokes equations with
degenerate viscosity, and a nonlinear third-order differential operator, with the quantum Bohm potential, and the damping terms. The global weak solutions to such system is shown by using the Faedo-Galerkin method and the compactness argument.
This system is also a very important approximation to the compressible Navier-Stokes equations. It will help us to prove the existence of global weak solutions to the compressible Navier-Stokes equations with degenerate viscosity
in three dimensional space.
\end{abstract}

\maketitle

\section{Introduction}
In this paper, we are interested in the existence of global weak solutions to  the barotropic compressible quantum Navier-Stokes equations  with damping terms
\begin{equation}
\begin{split}
\label{goal system}
&\rho_t+\Dv(\rho\u)=0,
\\&(\rho\u)_t+\Dv(\rho\u\otimes\u)+\nabla\rho^{\gamma}-\Dv(\rho\mathbb{D}\u)=-r_0\u-r_1\rho|\u|^2\u+\kappa\rho\nabla\left(\frac{\D\sqrt{\rho}}{\sqrt{\rho}}\right),
\end{split}
\end{equation}
with initial data as follows
\begin{equation}
\label{initial data}
\rho(0,x)=\rho_0(x),\;\;\;(\rho\u)(0,x)=\m_0(x)\quad\text{ in } \O,
\end{equation}
where $\rho$ is density, $\gamma>1$, $\u\otimes\u$ is the matrix with components $\u_i\u_j,$ $\mathbb{D}\u=\frac{1}{2}\left(\nabla\u+\nabla\u^T\right)$ is the sysmetic part of the velocity gradient,
 and $\O=\mathbb{T}^d$ is the $d-$dimensional torus, here $d=2$ or 3.
 The expression $\frac{\D\sqrt{\rho}}{\sqrt{\rho}}$ is called as Bohm potential which can be interpreted as a quantum potential.
The quantum Navier-Stokes equations have a lot of applications, in particular,  quantum semiconductors \cite{FZ}, weakly interacting Bose gases
\cite{Grant} and quantum trajectories of Bohmian mechanics \cite{W}. Recently some dissipative quantum
fluid models have been derived by J\"{u}ngel, see \cite{J 2012}. The damping terms
$$-r_0\u-r_1\rho|\u|^2\u$$
is motivated by the work of \cite{BD}. It allows us to recover the weak solutions to \eqref{goal system} by passing to the limits from the suitable approximation.
The most importance is that the existence
of solutions for the system \eqref{goal system} studied in the current paper is crucial to show the
existence of weak solutions for the Navier Stokes equations with degenerate
viscosity, see \cite{VY}. Models with these drag terms are also common in the literature, see \cite{BD,BD2006,FM}.

   \vskip0.3cm

When $r_0=r_1=\kappa=0$ in \eqref{goal system}, the system reduces to the compressible Navier-Stokes equations with degenerate viscosity $\mu(\rho)=\nu\rho$.
The existence of global weak solutions of such system has been a long standing open problem.
In the case $\gamma=2$ in 2D, this corresponds to the shallow water equations, where $\rho(t,x)$ stands for the height of the water at position $x$, and time $t$, and $\u(t,x)$ is the 2D velocity at the same position, and same time.
For  the constant viscosity case,
Lions in \cite{Lions} established the global existence of  renormalized solutions for $\gamma>\frac{9}{5}$, and
 Feireisl-Novotn\'{y}-Petzeltov\'{a} \cite{FNP} and Feireisl \cite{F04} extended the existence results to $\gamma>\frac{3}{2}$, and even to Navier-Stokes-Fourier system.
  The first tool of handling the degenerate viscosity is due to Bresch, Desjardins and Lin, see \cite{BDL}, where the authors deduced a new mathematical entropy to show the
   structure of the diffusion terms providing some regularity for the density. It was later extended for the case
with an additional quadratic friction term $r\rho|\u|\u$,
see Bresch-Desjardins \cite{BD,BD2006}.
Meanwhile, Mellet-Vasseur \cite{MV} deduced an estimate for proving the stability of smooth solutions for the compressible Navier-Stokes equations.

\vskip0.3cm

When $r_0=r_1=0$ in \eqref{goal system}, the system reduces to the so-called quantum Navier-Stokes equations. Up to our knowledge,
 there are no existence theorem of weak solutions for large data in any dimensional space. Compared to the degenerate compressible Navier-Stokes equations,
 we need to overcome the additional mathematical difficulty from the strongly nonlinear third- order differential operator.
  We have to mention that the Mellet-Vasseur type inequality does not hold for the quantum Navier-Stokes equations due to
the quantum potential. Thus, there are short of the suitable a priori estimates for proving the weak stability.  J\"{u}ngel \cite{J} used the test
 function of the form $\rho\varphi$ to handle the convection term, thus he proved the existence of such a particular
weak solution. In a very recent preprint, Gisclon-Violet \cite{GV} proved the existence of weak solutions to the quantum Navier-Stokes equations with
  singular pressure, where the authors adopt some arguments in \cite{Z} to make use of the cold pressure for compactness. Our methodology turns out to be very close to their paper. Actually,  the authors of \cite{GV} mention  that the existence can be obtained replacing the cold pressure by a drag force.
\vskip0.3cm

The existence of weak solutions to \eqref{goal system}, with the uniform bounds of Theorem \ref{main result 1}, is crucial for the existence of weak solutions to the compressible Navier-Stokes equations with degenerate viscosity in 3D, see \cite{VY}. In that work, we started from the weak solutions to \eqref{goal system}, that is, the main result of this current paper. Unfortunately, the version with the cold pressure proved in \cite{GV}, is not suitable for the  result in \cite{VY}.
On the approximation in \cite{VY},
 we need the terms $r_1\rho|\u|^2\u$ and $\kappa\rho(\frac{\D\sqrt{\rho}}{\sqrt{\rho}})$ for proving a key lemma. In particular, inequality \eqref{J inequality for weak solutions} is crucial to prove
the existence of weak solutions to the compressible Navier-Stokes equations in 3D. This estimate is from the term $\kappa\rho\nabla(\frac{\D\sqrt{\rho}}{\sqrt{\rho}}).$
\vskip0.3cm

We can deduce the following energy inequality for smooth solutions of \eqref{goal system}
\begin{equation}
\label{energy inequality for NS}
E(t)+\int_0^T\int_{\O}\rho|\mathbb{D}\u|^2\,dx\,dt+r_0\int_0^T\int_{\O}|\u|^2\,dx\,dt+r_1\int_0^T\int_{\O}\rho|\u|^4\,dx\,dt\leq E_0,
\end{equation}
where $$E(t)=E(\rho,\u)(t)=\int_{\O}\left(\frac{1}{2}\rho|\u|^2+\frac{1}{\gamma-1}\rho^{\gamma}+\frac{\kappa}{2}|\nabla\sqrt{\rho}|^2\right)\,dx,$$
and  $$E_0=E(\rho,\u)(0)=\int_{\O}\left(\frac{1}{2}\rho_0|\u_0|^2+\frac{1}{\gamma-1}\rho_0^{\gamma}+\frac{\kappa}{2}|\nabla\sqrt{\rho_0}|^2\right)\,dx.$$
However, we should point out that the above a priori estimate are not enough to show the
 stability of the solutions of \eqref{goal system}, in particular,
 for the compactness of $\rho^{\gamma}.$
we have the following Bresch-Desjardins entropy (see \cite{BD,BDL}) for providing more regularity of the density
\begin{equation}
\begin{split}
\label{BD entropy}
&\int_{\O}\left(\frac{1}{2}\rho|\u+\nabla\ln\rho|^2+\frac{\rho^{\gamma}}{\gamma-1}+\frac{\kappa}{2}|\nabla\sqrt{\rho}|^2-r_0\log\rho\right)\,dx+\int_0^T\int_{\O}|\nabla\rho^{\frac{\gamma}{2}}|^2\,dx\,dt
\\&+\int_0^T\int_{\O}\rho|\nabla\u-\nabla^T\u|^2\,dx\,dt+\kappa\int_0^T\int_{\O}\rho|\nabla^2\log \rho|^2\,dx\,dt
\\&\leq\int_{\O}\left(\rho_0|\u_0|^2+|\nabla\sqrt{\rho_0}|^2+\frac{\rho_0^{\gamma}}{\gamma-1}+\frac{\kappa}{2}|\nabla\sqrt{\rho_0}|^2-r_0\log_{-}\rho_0\right)\,dx+C,
\end{split}
\end{equation}
where $C$ is bounded by the initial energy, $\log_{-}g=\log\min(g,1).$

Thus, the initial data should be given in such a way
\begin{equation}
\begin{split}
\label{initial condition}
&\rho_0\in L^{\gamma}(\O),\;\;\;\rho_0\geq 0,\;\;\; \nabla\sqrt{\rho_0}\in L^2(\O),\;\;-\log_{-}\rho_0\in L^1(\O),\\
&\m_0\in L^1(\O),\;\;\m_0=0\;\;\text{ if } \;\rho_0=0,\;\;\frac{|\m_0|^2}{\rho_0}\in L^1(\O).
\end{split}
\end{equation}

We define the weak solution $(\rho,\u)$ to the initial value problem \eqref{goal system} in the following sense: for any $t\in[0,T]$,
\begin{itemize}
\item  \eqref{initial data} holds in $\mathcal{D'}(\O)$,
\item \eqref{energy inequality for NS} and \eqref {BD entropy} hold for almost every $t\in[0,T]$,
\item \eqref{goal system} holds in $\mathcal{D'}((0,T)\times\O))$ and the following is satisfied\\
\begin{equation*}
\begin{split}
& \rho\in L^{\infty}(0,T;L^{\gamma}(\O)),\quad\quad\quad\sqrt{\rho}\u\in L^{\infty}(0,T;L^2(\O)),\\
& \nabla\sqrt{\rho}\in L^{\infty}(0,T;L^2(\O)),\;\;\;\;\;\;\nabla\rho^{\frac{\gamma}{2}}\in L^2(0,T;L^2(\O)),
\\&
\sqrt{\rho}\mathbb{D}\u\in L^2(0,T;L^2(\O)),\;\;\;\;\;\;\sqrt{\rho}\nabla\u\in L^2(0,T;L^2(\O)),
\\&\rho^{\frac{1}{4}}\u\in L^4(0,T;L^4(\O)),\;\;\;\;\;\;\;\;\;\u\in L^2(0,T;L^2(\O)),
\\&\sqrt{\rho}|\nabla^2\log\rho|\in L^2(0,T;L^2(\O)).
\end{split}
\end{equation*}
 \end{itemize}

The following is our main result.
\begin{Theorem}
\label{main result 1} If the initial data satisfy \eqref{initial condition},
 there exists a weak solution $(\rho,\u)$ to \eqref{goal system}-\eqref{initial data} for any $\gamma>1$, any $T>0$,
in particular, the weak solution $(\rho,\u)$ satisfies energy inequality \eqref{energy inequality for NS}, BD-entropy \eqref{BD entropy} and  the following inequality:
\begin{equation}
\label{J inequality for weak solutions}
\kappa^{\frac{1}{2}}\|\sqrt{\rho}\|_{L^2(0,T;H^2(\O))}+\kappa^{\frac{1}{4}}\|\nabla\rho^{\frac{1}{4}}\|_{L^4(0,T;L^{4}(\O))}\leq C,
\end{equation}
where $C$  only depends on the initial data. Moreover, the weak solution $(\rho,\u)$ has the following properties
\begin{equation}
\begin{split}
\label{property}
&\rho\u\in C([0,T];L^{\frac{3}{2}}_{weak}(\O)),\quad(\sqrt{\rho})_t\in L^2((0,T)\times\O);
\end{split}
\end{equation}
If we use $(\rho_{\kappa},\u_{\kappa})$ to denote the weak solution for $\kappa>0$, then
\begin{equation}
\begin{split}
\label{property-2}
&\sqrt{\rho_{\kappa}}\u_{\kappa}\to \sqrt{\rho}\u\,\,\,\,\text{strongly in } L^2((0,T)\times\O),\;\;\text{as } \;\kappa\to0,
\end{split}
\end{equation}
where $(\rho,\u)$ in \eqref{property-2} is a weak solution to \eqref{goal system}-\eqref{initial data}  with $\kappa=0.$  We remark the metric space $C([0,T];L^{\frac{3}{2}}_{weak}(\O))$ of function
$f:[0,T]\to L^{\gamma}(\O)$  which are continuous with respect to the weak topology.
\end{Theorem}

\begin{Remark}
\label{remark on properties} We will use \eqref{J inequality for weak solutions}-
\eqref{property} in \cite{VY} to prove the weak solutions to \eqref{goal system} with $r_0=r_1=\kappa=0.$ In fact, inequality \eqref{J inequality for weak solutions} is very crucial to prove a key lemma in \cite{VY}.
\end{Remark}
\begin{Remark}
The existence result contains the case
 with $\kappa=0$,  which can be obtained as the limit when $\kappa>0$ goes to 0 in (\ref{goal system}), by  standard compactness analysis.
\end{Remark}

\begin{Remark}
\label{remark on weak formulation} The weak formulation reads as
\begin{equation}
\begin{split}
\label{weak formulation of NSK}
&\int_{\O}\rho\u\cdot\psi\,dx|_{t=0}^{t=T}-\int_{0}^{T}\int_{\O}\rho\u\psi_t\,dx\,dt
-\int_{0}^{T}\int_{\O}\rho\u\otimes\u:\nabla \psi\,dx\,dt\\&-\int_{0}^{T}\int_{\O}\rho^{\gamma}\Dv\psi\,dx\,dt
-\int_0^{T}\int_{\O}\rho\mathbb{D}\u:\nabla\psi\,dx\,dt
\\& =-r_0\int_{0}^{T}\int_{\O}\u\psi\,dx\,dt-r_1\int_0^T\int_{\O}\rho|\u|^2\u\psi\,dx\,dt-2\kappa\int_0^T\int_{\O}\D\sqrt{\rho}\nabla\sqrt{\rho}\psi\;dx\;dt
\\&-\kappa\int_0^T\int_{\O}\D\sqrt{\rho}\sqrt{\rho}\Dv\psi\,dx\,dt.
\end{split}
\end{equation}
for any test function $\psi.$
\end{Remark}

\section{Faedo-Galerkin approximation}
In this section, we construct the solutions to the approximation scheme by Faedo-Galerkin method. Motivated by the work of Feireisl-Novotn\'{y}-Petzeltov\'{a} \cite{FNP} and Feireisl \cite{F04}, we proceed similarly as in J\"{u}ngel \cite{J}.
We introduce a finite dimensional space $X_N=\text{span}\{e_1,e_2,....,e_N\}$, where $N\in \mathbb{N}$, each $e_i$ be an orthonormal basic of $L^2(\O)$ which is also an orthogonal basis of $H^2(\O).$ We notice that
$\u\in C^0([0,T];X_N)$ is given by
$$\u(t,x)=\sum_{i=1}^{N}\lambda_{i}(t)e_{i}(x),\;\;\;(t,x)\in [0,T]\times\O,$$
for some functions $\lambda_i(t)$, and the norm of $\u$ in $C^0([0,T];X_N)$ can be written as
$$\|\u\|_{C^0([0,T];X_n)}=\sup_{t\in [0,T]}\sum_{i=1}^{N}|\lambda_i(t)|.$$
And hence, $\u$ can be bounded in $C^0([0,T];C^k(\O))$ for any $k\geq 0$, thus
$$\|\u\|_{C^0([0,T];C^k(\O))}\leq C(k)\|\u\|_{C^0([0,T];L^2(\O))}.$$
For any given $\u \in C^0([0,T];X_N),$ by the classical theory of parabolic equation, there exists a classical solution $\rho(t,x)\in C^1([0,T];C^3(\O))$ to the following approximated system
\begin{equation}
\label{approximated system parabolic equation}
\rho_t+\Dv(\rho\u)=\varepsilon\D\rho,\quad\rho(0,x)=\rho_0(x)\quad\text{ in } (0,T)\times \O
\end{equation}
with the initial data
\begin{equation}
\label{regulation intial data}\rho(0,x)=\rho_0(x)\geq\nu>0,\quad\text{ and }\rho_0(x)\in C^{\infty}(\O),
\end{equation}
where $\nu>0$ is a constant.

We should remark that this solution $\rho(t,x)$ satisfies the following inequality
\begin{equation}
\label{below and above estimate on density}
\inf_{x\in \O}\rho_0(x)\exp^{-\int_0^T\|\Dv\u\|_{L^{\infty}(\O)}\,ds}\leq \rho(t,x)\leq
\sup_{x\in \O}\rho_0(x)\exp^{\int_0^T\|\Dv\u\|_{L^{\infty}(\O)}\,ds}
\end{equation}
for all $(t,x)$ in $(0,T)\times\O.$ By \eqref{regulation intial data} and \eqref{below and above estimate on density}, there exists a constant $\theta_0>0$ such  that
\begin{equation}
\label{lower and upper bounds for density}
0< \theta_0\leq \rho(t,x)\leq \frac{1}{\theta_0}\quad\quad\text{ for } (t,x)\in (0,T)\times \O.
\end{equation}
Thus, we can introduce a linear continuous operator $S:\;\; C^0([0,T];X_N)\to C^0([0,T];C^k(\O))$ by $S(\u)=\rho$,
and \begin{equation}
\label{continuous on S}
\|S(\u_1)-S(\u_2)\|_{C^0([0,T];C^k(\O))}\leq C(N,k)\|\u_1-\u_2\|_{C^0([0,T];L^2(\O))}
\end{equation}
 for any $k\geq 1.$

The Faedo-Galerkin approximation for the weak formulation of the momentum balance is as follows
\begin{equation}
\begin{split}
\label{approximation momentum balance}
&\int_{\O}\rho\u(T)\varphi\,dx-\int_{\O}\m_0\varphi\,dx+\mu\int_0^T\int_{\O}\D\u\cdot\D\varphi\,dx\,dt-\int_0^T\int_{\O}(\rho\u\otimes\u):\nabla\varphi\,dx\,dt
\\&+\int_0^T\int_{\O}2\rho\mathbb{D}\u:\nabla\varphi\,dx\,dt-\int_0^T\int_{\O}\rho^{\gamma}\nabla\varphi\,dx\,dt+\eta\int_0^T\int_{\O}\rho^{-10}\nabla\varphi\;dx\,dt
\\&+\varepsilon\int_0^T\int_{\O}\nabla\rho\cdot\nabla\u\varphi\,dx\,dt=-r_0\int_0^T\int_{\O}\u\varphi\,dx\,dt-r_1\int_0^T\int_{\O}\rho|\u|^2\u\varphi\,dx\,dt
\\&-2\kappa\int_0^T\int_{\O}\D\sqrt{\rho}\nabla\sqrt{\rho}\psi\;dx\;dt
-\kappa\int_0^T\int_{\O}\D\sqrt{\rho}\sqrt{\rho}\Dv\psi\,dx\,dt+\delta\int_0^T\int_{\O}\rho\nabla\D^{9}\rho\varphi\,dx\,dt,
\end{split}
\end{equation}
for any test function $\varphi\in X_N$.  The extra terms $\eta\nabla\rho^{-10}$ and $\delta\rho\nabla\D^9\rho$ are necessary to keep the density bounded, and bounded away from zero for all time. This enables us to take
$\frac{\nabla\rho}{\rho}$ as a test function to derive the Bresch-Desjardins entropy.

 To solve \eqref{approximation momentum balance}, we follow the same arguments as  in \cite{FNP, F04, J} and
introduce the following operators, given the density function $\rho(t,x)\in L^1(\O)$ with $\rho\geq\underline{\rho}>0,$ here we choose $\underline{\rho}=\theta_0.$
We define
$$\mathfrak{M}[\rho(t),\cdot]:X_N\to X_N^*,\quad<\mathfrak{M}[\rho]\u,\w>=\int_{\O}\rho\u\cdot\w\,dx,\;\;\text{ for }\u,\w\in X_N.$$
We can show that $\mathfrak{M}[\rho]$ is invertible $$\|\mathfrak{M}^{-1}(\rho)\|_{L(X_N^{*},X_N)}\leq \underline{\rho}^{-1},$$
where $L(X_N^{*},X_N)$ is the set of all bounded linear mappings from $X_N^{*}$ to $X_N$.
  It is Lipschitz continuous in the following sense
\begin{equation}
\label{continuous on M}
\|\mathfrak{M}^{-1}(\rho_1)-\mathfrak{M}^{-1}(\rho_2)\|_{L(X_N^{*},X_N)}\leq C(N,\underline{\rho})\|\rho_1-\rho_2\|_{L^1(\O)}
\end{equation}
for
any $\rho_1$ and $\rho_2$ from the following set $$N_{\nu}=\{\rho\in L^1(\O)| \;\;\inf_{x\in \O}\rho\geq \nu>0.\}$$
 For more details, we refer the readers to \cite{FNP,F04,J}.

We are looking for $\u_n\in C([0,T];X_n)$ solution of the following nonlinear integral equation
\begin{equation}
\label{integral equation}
\u_N(t)=\mathfrak{M}^{-1}(S(\u_N))(t)\left(\mathfrak{M}[\rho_0](\u_0)+\int_0^T\mathfrak{N}(S(\u_N),\u_N)(s)\,ds\right),
\end{equation}
where
\begin{equation*}
\begin{split}
\mathfrak{N}(S(\u_N),\u_N)&=-\Dv(\rho\u_N\otimes\u_N)+\Dv(\rho\mathbb{D}\u_N)+\mu\D^2\u_N-\varepsilon\nabla\rho\cdot\nabla\u_N+\eta\nabla\rho^{-10}-\nabla\rho^{\gamma}
\\&-r_0\u_N-r_1\rho|\u_N|^2\u_N+\kappa\rho\nabla\left(\frac{\D\sqrt{\rho}}{\sqrt{\rho}}\right)+\delta\rho\nabla\D^9\rho,
\end{split}
\end{equation*}
$\rho=S(\u_N)$.

 Thanks to \eqref{continuous on S} and \eqref{continuous on M}, we can apply a fixed point argument to solve the nonlinear equation \eqref{integral equation}  on a short time interval $[0,T^{*}]$ for $T^{*}\leq T,$ in the space $C^0([0,T^{*}];X_N).$ Thus, there exists a local-in-time solution $(\rho_N,\u_N)$ to \eqref{approximated system parabolic equation}, \eqref{integral equation}. Observe that $L^2-$ norm and
$C^2-$norm are equivalent on $X_N$.

Differentiating \eqref{approximation momentum balance} with respect to time $t$ and taking $\varphi=\u_N$, we have the following energy balance
\begin{equation}
\label{enery equality}
\begin{split}&\frac{d}{dt}E(\rho_N,\u_N)+\mu\int_{\O}|\D\u_N|^2\,dx+\int_{\O}\rho_N|\mathbb{D}\u_N|^2\,dx+\varepsilon\delta\int_{\O}|\D^5\rho_N|^2\,dx
\\&+\varepsilon\int_{\O}|\nabla\rho_N^{\frac{\gamma}{2}}|^2\,dx+\varepsilon\eta\int_{\O}|\nabla\rho_N^{-5}|^2\,dx
+r_0\int_{\O}|\u_N|^2\,dx
+r_1\int_{\O}\rho_N|\u_N|^4\,dx
\\&+\kappa\varepsilon\int_{\O}\rho_N|\nabla^2\log\rho_N|^2\,dx=0,
\end{split}
\end{equation}
on $[0,T^{*}],$
where $$E(\rho_N,\u_N)=\int_{\O}\left(\frac{1}{2}\rho_N|\u_N|^2+\frac{\rho_N^{\gamma}}{\gamma-1}+\frac{\eta}{10}\rho_N^{-10}+\frac{\kappa}{2}|\nabla\sqrt{\rho_N}|^2+\frac{\delta}{2}|\nabla\D^4\rho_N|^2\right)\,dx,$$
and
$$E_0(\rho_N,\u_N)=\int_{\O}\left(\frac{1}{2}\rho_0|\u_0|^2+\frac{\rho_0^{\gamma}}{\gamma-1}+\frac{\eta}{10}\rho_0^{-10}+\frac{\kappa}{2}|\nabla\sqrt{\rho_0}|^2+\frac{\delta}{2}|\nabla\D^4\rho_0|^2\right)\,dx.$$

Here we used the identity $$2\rho_N\nabla(\frac{\D\sqrt{\rho_N}}{\sqrt{\rho_N}})=\Dv\left(\rho_N\nabla^2(\log\rho_N)\right)$$
to yield
$$\int_{\O}\frac{\D\sqrt{\rho_N}}{\sqrt{\rho_N}}\D\rho_N\,dx=-\int_{\O}\rho_N\nabla\log\rho_N\cdot\nabla\left(\frac{\D\sqrt{\rho_N}}{\sqrt{\rho_N}}\right)\,dx=\frac{1}{2}\int_{\O}\rho_N|\nabla^2\log\rho_N|^2\,dx.$$
Energy equality \eqref{enery equality} yields
\begin{equation}
\label{high estimate on velocity}
\int_0^{T^{*}}\|\D\u_N\|_{L^2}^2\,dt\leq E_0(\rho_n,\u_N)<\infty.
\end{equation}
Due to $\text{dim} X_N<\infty$ and \eqref{below and above estimate on density}, there exists a constant $\theta_0>0$ such  that
\begin{equation}
\label{positive density1}
0<\theta_0\leq \rho_N(t,x)\leq \frac{1}{\theta_0}
\end{equation}
for all $t\in(0,T^{*}).$ However, this $\theta_0$ depends on $N$ and it is the same to $\theta_0$ in \eqref{lower and upper bounds for density}.
Energy equality \eqref{enery equality} gives us
$$\sup_{t\in(0,T^{*})}\int_0^{T^{*}}\rho_N|\u_N|^2\,dx\leq E_0(\rho_N,\u_N)<\infty,$$
and $$\int_0^{T^{*}}\rho_N|\mathbb{D}\u_N|^2\,dx\,dt\leq E_0(\rho_N,\u_N)<\infty,$$
which, together with \eqref{high estimate on velocity}, \eqref{positive density1}, implies
\begin{equation}
\label{W estimate on velocity}
\sup_{0\in(0,T^{*})}\left(\|\u_N\|_{L^{\infty}}+\|\nabla\u_N\|_{L^{\infty}}+\|\D\u_N\|_{L^{\infty}}\right)\leq C(E_0(\rho_N,\u_N),N),
\end{equation}
where we used a fact that the equivalence of $L^2$ and $L^{\infty}$ on $X_N$.
By \eqref{continuous on S}, \eqref{continuous on M}, \eqref{positive density1} and \eqref{W estimate on velocity}, repeating our above arguments many times, we can extend $T^{*}$ to $T$. Thus there exists a solution $(\rho_N,\u_N)$ to \eqref{approximated system parabolic equation}, \eqref{integral equation} for any $T>0.$

Here we need to state the following lemma due to J\"{u}ngel \cite{J}:
\begin{Lemma}
\label{Lemma of J inequality}
For any smooth positive function $\rho(x)$, we have
$$\int_{\O}\rho|\nabla^2\log\rho|^2\,dx\geq \frac{1}{7}\int_{\O}|\nabla^2\sqrt{\rho}|^2\,dx$$
and
$$\int_{\O}\rho|\nabla^2\log\rho|^2\,dx\geq \frac{1}{8}\int_{\O}|\nabla\rho^{\frac{1}{4}}|^4\,dx.$$
\end{Lemma}
\begin{proof}
The above inequality of Lemma is firstly proved by J\"{u}ngel \cite{J}. Here, we give a quick proof.\\
We notice \begin{equation}
\label{equality for density}
\sqrt{\rho}\cdot\nabla^2\log\sqrt{\rho}=\sqrt{\rho}\cdot\nabla(\frac{\nabla\sqrt{\rho}}{\sqrt{\rho}})=\nabla^2\sqrt{\rho}-\frac{\nabla\sqrt{\rho}\otimes\nabla\sqrt{\rho}}{\sqrt{\rho}},
\end{equation}
thus
\begin{equation*}
\begin{split}
\int_{\O}\rho|\nabla^2\log\sqrt\rho|^2\,dx&=\int_{\O}|\nabla^2\sqrt{\rho}|^2\,dx+\int_{\O}|2\nabla\rho^{\frac{1}{4}}|^4\,dx-2\int_{\O}\nabla^2\sqrt{\rho}\cdot\frac{\nabla\sqrt{\rho}\otimes\nabla\sqrt{\rho}}{\sqrt{\rho}},
\\&=A+B-I,
\end{split}
\end{equation*}
For $I$, we control it as follows
\begin{equation*}
\begin{split}&I=2\int_{\O}\nabla^2\sqrt{\rho}\cdot\left(\frac{\nabla\sqrt{\rho}}{\sqrt{\rho}}\otimes\nabla\sqrt{\rho}\right)\,dx\\
&=-2\int_{\O}\frac{|\nabla\sqrt{\rho}|^2}{\sqrt{\rho}}\Delta\log\sqrt{\rho}\,dx-2\int_{\O}\nabla^2\sqrt{\rho}\cdot\frac{\nabla\sqrt{\rho}\otimes\nabla\sqrt{\rho}}{\sqrt{\rho}}\,dx.
\end{split}
\end{equation*}
Hence:
\begin{equation*}
\begin{split}&2I=-2\int_{\O}\frac{|\nabla\sqrt{\rho}|^2}{\sqrt{\rho}}\Delta\log\sqrt{\rho}\,dx\leq 2\sqrt{3 BD},
\end{split}
\end{equation*}
where  $D=\int_{\O}\rho|\nabla^2\log\sqrt\rho|^2\,dx$, and hence
$$A+B=D+I\leq (1+6)D+\frac{1}{8}B,$$
and thus, $$\frac{1}{7}A+\frac{1}{8}B\leq D.$$
So we proved this lemma.
\end{proof}

By \eqref{enery equality}, we have
\begin{equation*}
E(\rho_N,\u_N)\leq E_0(\rho_N,\u_N),
\end{equation*}
this gives us
\begin{equation*}
\|\rho_N\|_{L^{\infty}(0,T;H^9(\O))}\leq C(E_0(\rho_N,\u_N),\delta),
\end{equation*}
this, together with \eqref{positive density1}, gives us that the density $\rho(t,x)$ is a positive smooth function for all $(t,x).$
We also notice that $$\kappa\varepsilon\int_0^T\int_{\O}\rho_N|\nabla^2\log\rho_N|^2\,dx\,dt\leq E_0(\rho_N,\u_N)<\infty.$$
By Lemma \ref{Lemma of J inequality}, we have the following uniform estimate:
\begin{equation}
\label{estimate on density 1/2}
(\kappa\varepsilon)^{\frac{1}{2}}\|\sqrt{\rho_N}\|_{L^2(0,T;H^2(\O))}+(\kappa\varepsilon)^{\frac{1}{4}}\|\nabla\rho_N^{\frac{1}{4}}\|_{L^4(0,T;L^{4}(\O))}\leq C,
\end{equation}
where the constant $C>0$ is independent of $N$.

To conclude this part, we have the following lemma on the approximate solutions $(\rho_N,\u_N)$:
\begin{Proposition}
\label{Lemma of existence on first level}
Let $(\rho_N,\u_N)$ be the solution of \eqref{approximated system parabolic equation}, \eqref{integral equation} on $(0,T)\times\O$ constructed above, then we have the following energy inequality
\begin{equation}
\label{enery inequality for the first level}
\begin{split}&\sup_{t\in(0,T)}E(\rho_N,\u_N)+\mu\int_0^T\int_{\O}|\D\u_N|^2\,dx\,dt+\int_0^T\int_{\O}\rho_N|\mathbb{D}\u_N|^2\,dx\,dt+\varepsilon\delta\int_0^T\int_{\O}|\D^5\rho_N|^2\,dx\,dt
\\&+\varepsilon\int_0^T\int_{\O}|\nabla\rho_N^{\frac{\gamma}{2}}|^2\,dx\,dt+\varepsilon\eta\int_0^T\int_{\O}|\nabla\rho_N^{-5}|^2\,dx\,dt
+r_0\int_0^T\int_{\O}|\u_N|^2\,dx\,dt
\\&+r_1\int_0^T\int_{\O}\rho_N|\u_N|^4\,dx\,dt
+\kappa\varepsilon\int_0^T\int_{\O}\rho_N|\nabla^2\log\rho_N|^2\,dx\,dt\leq E_0(\rho_N,\u_N),
\end{split}
\end{equation}
where $$E(\rho_N,\u_N)=\int_{\O}\left(\frac{1}{2}\rho_N|\u_N|^2+\frac{\rho_N^{\gamma}}{\gamma-1}+\frac{\eta}{10}\rho_N^{-10}+\frac{\kappa}{2}|\nabla\sqrt{\rho_N}|^2+\frac{\delta}{2}|\nabla\D^4\rho_N|^2\right)\,dx.$$
Moreover, we have the following uniform estimate:
\begin{equation}
\label{J inequality N}
(\kappa\varepsilon)^{\frac{1}{2}}\|\sqrt{\rho_N}\|_{L^2(0,T;H^2(\O))}+(\kappa\varepsilon)^{\frac{1}{4}}\|\nabla\rho_N^{\frac{1}{4}}\|_{L^4(0,T;L^{4}(\O))}\leq C,
\end{equation}
where the constant $C>0$ is independent of $N$.\\
In particular, we have the following estimates
\begin{equation}
\label{L1}
\sqrt{\rho_N}\u_N\in L^{\infty}(0,T;L^2(\O)),\sqrt{\rho_N}\mathbb{D}\u_N\in L^2((0,T)\times\O),\sqrt{\mu}\D\u_N\in L^2((0,T)\times\O),
\end{equation}
\begin{equation}
\label{L2}
\sqrt{\varepsilon\delta}\D^5\rho_N\in L^2((0,T)\times\O),\sqrt{\delta}\rho_N\in L^{\infty}(0,T;H^9(\O)),\sqrt{\kappa}\sqrt{\rho_N}\in L^{\infty}(0,T;H^1(\O)),
\end{equation}
\begin{equation}
\label{L3}\nabla\rho_N^{\frac{\gamma}{2}}\in L^2((0,T)\times\O),\;\;\;\;
\rho_N^{-1}\in L^{\infty}(0,T;L^{10}(\O)),\sqrt{\varepsilon\eta}\nabla\rho_N^{-5}\in L^2((0,T)\times\O),
\end{equation}
\begin{equation}
\label{L4}
\u_N\in L^2((0,T)\times\O),\;\;\;\rho_N^{\frac{1}{4}}\u_N\in L^4((0,T)\times\O).
\end{equation}
\end{Proposition}
Based on above estimates, we have the following estimates uniform on $N$:
\begin{Lemma}
\label{Lemma for uniofrm estimates N}
The following estimates hold for any fixed positive constants $\varepsilon$, $\mu$, $\eta$ and $\delta$:
\begin{equation}
\label{estimate for density N}
\begin{split}
&\|(\sqrt{\rho_{N}})_t\|_{L^2((0,T)\times\O)}+\|\sqrt{\rho_{N}}\|_{L^{2}(0,T;H^2(\O))}\leq K,\\&
\|(\rho_{N})_t\|_{L^2((0,T)\times\O)}+\|\rho_{N}\|_{L^2(0,T;H^{10}(\O))}\leq K,
\end{split}
\end{equation}
\begin{equation}
\label{estimate for mass N}
\|(\rho_{N}\u_{N})_t\|_{L^2(0,T;H^{-9}(\O))}+\|\rho_{N}\u_{N}\|_{L^{2}((0,T)\times\O)}\leq K,
\end{equation}
\begin{equation}
\label{nabla of mass}
\nabla(\rho_N\u_N)\;\;\text{ is uniformly  bounded in } L^{4}(0,T;L^{\frac{6}{5}}(\O))+L^2(0,T;L^{\frac{3}{2}}(\O)).
\end{equation}
\begin{equation}
\label{estimate presure N}
\|\rho_{N}^{\gamma}\|_{L^{\frac{5}{3}}((0,T)\times\O)}\leq K,
\end{equation}
\begin{equation}
\label{estimate cold presure N}
\|\rho_{N}^{-10}\|_{L^{\frac{5}{3}}((0,T)\times\O)}\leq K,
\end{equation}
where $K$ is independent of $N$, depends on $\varepsilon,\,\mu,\,\eta$ and $\delta$.
\end{Lemma}

\begin{proof} By \eqref{J inequality N}, we have
$$\|\sqrt{\rho_{N}}\|_{L^2(0,T;H^2(\O))}\leq C.$$
We notice that
\begin{equation*}
\begin{split}
(\rho_{N})_t&=-\rho_{N}\Dv\u_{N}-\nabla\rho_{N}\cdot\u_{N}
\\&=-(4\nabla\rho_{N}^{\frac{1}{4}})(\rho_{N}^{\frac{1}{4}}\u_{N})(\rho_{N}^{\frac{1}{2}})-\sqrt{\rho_{N}}\sqrt{\rho_{N}}\Dv\u_{N},
\end{split}
\end{equation*}
which gives us \begin{equation*}
\begin{split}
&\|(\rho_{N})_t\|_{L^2((0,T)\times\O)}
\leq 4\|\nabla\rho_{N}^{\frac{1}{4}}\|_{L^4((0,T)\times\O)}\|\rho_{N}^{\frac{1}{4}}\u_{N}\|_{L^4((0,T)\times\O)}\|\rho_N^{\frac{1}{2}}\|_{L^{\infty}((0,T)\times\O)}
\\&\quad\quad\quad\quad\quad\quad\quad\quad+\|\sqrt{\rho_{N}}\|_{L^{\infty}((0,T)\times\O)}\|\sqrt{\rho_{N}}\nabla\u_{N}\|_{L^2((0,T)\times\O)},
\end{split}
\end{equation*}
 thanks to \eqref{L1}-\eqref{L4} and Sobelov inequality.

Meanwhile, we have
\begin{equation*}
\begin{split}
2(\sqrt{\rho_{N}})_t&=-\sqrt{\rho_{N}}\Dv\u_{N}-2\nabla\sqrt{\rho_{N}}\cdot\u_{N}
\\&=-\sqrt{\rho_{N}}\Dv\u_{N}-8\nabla\rho_{N}^{\frac{1}{4}}\rho_{N}^{\frac{1}{4}}\u_{N},
\end{split}
\end{equation*}
which yields $(\sqrt{\rho_{N}})_t$ is bounded in $L^2((0,T)\times\O).$

Here we claim that $(\rho_{N}\u_{N})_t$ is bounded in $L^2(0,T;H^{-9}(\O))$. By
\begin{equation*}
\begin{split}
(\rho_{N}\u_{N})_t&=-\Dv(\rho_{N}\u_{N}\otimes\u_{N})-\nabla\rho_{N}^{\gamma}+\eta\nabla\rho_{N}^{-10}+\mu\D^2\u_{N}+\Dv(\rho_{N}\mathbb{D}\u_{N})-r_0\u_{N}\\
&-r_1\rho_{N}|\u_{N}|^2\u_{N}
+\varepsilon\nabla\rho_{N}\cdot\nabla\u_{N}+\kappa\rho_{N}\nabla\left(\frac{\D\sqrt{\rho_{N}}}{\sqrt{\rho_{N}}}\right)+\delta\rho_{N}\nabla\D^9\rho_{N},
\end{split}
\end{equation*}
we can show the claim by the above estimates.

And $$\|\rho_{N}\u_{N}\|_{L^{2}((0,T)\times\O)}\leq \|\rho_{N}^{\frac{3}{4}}\|_{L^{\infty}(0,T;L^{4}(\O))}\|\rho_{N}^{\frac{1}{4}}\u_{N}\|_{L^{4}((0,T)\times\O)}\leq K,$$
where we used Sobelov inequality and \eqref{J inequality N}.
Thus we have \eqref{estimate for mass N}.

We calculate
\begin{equation*}
\begin{split}
\nabla(\rho_N\u_N)=\nabla\sqrt{\rho_N}\rho_n^{\frac{1}{4}}\u_N\rho^{\frac{1}{4}}+\sqrt{\rho_N}\sqrt{\rho_N}\nabla\u_N,
\end{split}
\end{equation*}
it allows us to have \eqref{nabla of mass}.
For any given $\varepsilon>0$, we have $$\|\nabla\rho_{N}^{\frac{\gamma}{2}}\|_{L^2((0,T)\times\O)}\leq K,$$
which gives us $$\|\rho_{N}^{\gamma}\|_{L^1(0,T;L^3(\O))}\leq K.$$
Notice $$\rho_{N}^{\gamma}\in L^{\infty}(0,T;L^1(\O)),$$ we apply H\"{o}lder inequality to have
$$\|\rho_{N}^{\gamma}\|_{L^{\frac{5}{3}}((0,T)\times\O)}\leq \|\rho_{N}^{\gamma}\|_{L^{\infty}(0,T;L^1(\O))}^{\frac{2}{5}}\|\rho_{N}^{\gamma}\|_{L^1(0,T;L^3(\O))}^{\frac{3}{5}}\leq K.$$
Similarly, we can show \eqref{estimate cold presure N}.
\end{proof}

Applying Aubin-Lions Lemma and Lemma \ref{Lemma for uniofrm estimates N},
we conclude
\begin{equation}
\label{strong convergence of density-1}\rho_{N}\to\rho\;\;\text{ strongly in } L^2(0,T;H^9(\O)),\;\text{ weakly in } L^2(0,T;H^{10}(\O)),
\end{equation}
$$\sqrt{\rho_{N}}\to\sqrt{\rho}\;\;\text{ strongly in } L^2(0,T;H^1(\O)),\;\text{ weakly in } L^2(0,T;H^{2}(\O))$$
and \begin{equation}
\label{strong convergence of mass-1}\rho_{N}\u_{N}\to\rho\u\;\;\text{ strongly in } L^2((0,T)\times\O).
\end{equation}
We notice that $\u_{N}\in L^2((0,T)\times\O),$ thus, $$\u_{N}\to \u\;\;\text{ weakly in } L^2((0,T)\times\O).$$
Thus, we can pass into the limits for term $\rho_{N}\u_{N}\otimes\u_{N}$ as follows
$$\rho_{N}\u_{N}\otimes\u_{N}\to\rho\u\otimes\u$$
in the distribution sense.

Here we state the following lemma on the convergence of $\rho_{N}|\u_{N}|^2\u_{N}$.
\begin{Lemma}
\label{Lemma of convergence on drag term}
When $N\to\infty$, we have
$$\rho_{N}|\u_{N}|^2\u_{N}\to\rho|\u|^2\u\quad\text{ strongly in } L^1(0,T;L^1(\O)).$$
\end{Lemma}

\begin{proof}
Fatou's lemma yields
$$\int_{\O}\rho|\u|^4\,dx\leq \int_{\O}\lim\inf\rho_{N}|\u_{N}|^4\,dx\leq \lim\inf\int_{\O}\rho_{N}|\u_{N}|^4\,dx,$$
and hence $\rho|\u|^4$ is in $L^1(0,T;L^1(\O)).$

By \eqref{strong convergence of density-1} and \eqref{strong convergence of mass-1}, we have, up to a subsequence, such that
$$\rho_{N}\to \rho(t,x)\quad\text{ a.e.}$$
and $$\rho_{N}\u_{N}\to\rho\u\quad\text{ a.e.}$$
 Thus, for almost every $(t,x)$ such that when $\rho_{N}(t,x)\neq 0$, we have
$$\u_{N}=\frac{\rho_{N}\u_{N}}{\rho_{N}}\to \u.$$
For almost every $(t,x)$ such that $\rho_{N}(t,x)=0,$
then
$$\rho_{N}|\u_{N}|^2\u_{N}\chi_{|\u_{N}|\leq M}\leq M^3\rho_{N}=0=\rho|\u|^2\u\chi_{|\u|\leq M}.$$
Hence,
$\rho_{N}|\u_{N}|^2\u_{N}\chi_{|\u_{N}|\leq M}$ converges to $ \rho|\u|^2\u\chi_{|\u|\leq M}$ almost everywhere for $(t,x).$
 Meanwhile, $
\rho_{N}|\u_{N}|^2\u_{N}\chi_{|\u_{N}|\leq M}$ is uniformly bounded in $L^{\infty}(0,T;L^2(\O))$ thanks to
\eqref{L2}.

 The dominated convergence theorem gives us
\begin{equation}
\label{strong convergence cut M}
\rho_{N}|\u_{N}|^2\u_{N}\chi_{|\u_{N}|\leq M} \to \rho|\u|^2\u\chi_{|\u|\leq M}\quad \text{strongly in } L^1(0,T;L^1(\O)).
\end{equation}

For any $M>0$, we have
\begin{equation}
\begin{split}&
\int_0^T\int_{\O}\left|
\rho_{N}|\u_{N}|^2\u_{N}-\rho|\u|^2\u\right|\,dx\,dt
\\&\leq
\int_0^T\int_{\O}\left|
\rho_{N}|\u_{N}|^2\u_{N}\chi_{|\u_{N}|\leq M}-\rho|\u|^2\u\chi_{|\u|\leq M}\right|\,dx\,dt
\\&+2
\int_0^T\int_{\O}\rho_{N}|\u_{N}|^3\chi_{|\u_{N}|\geq M}\,dx\,dt+2
\int_0^T\int_{\O}
\rho|\u|^3\chi_{|\u|\geq M}\,dx\,dt
\\&\leq
\int_0^T\int_{\O}\left|
\rho_{N}|\u_{N}|^2\u_{N}\chi_{|\u_{N}|\leq M}-\rho|\u|^2\u\chi_{|\u|\leq M}\right|\,dx\,dt
\\&+\frac{2}{M}
\int_0^T\int_{\O}\rho_{N}|\u_{N}|^4\,dx\,dt+\frac{2}{M}\int_0^T\int_{\O}\rho|\u|^4\,dx\,dt.
\end{split}
\end{equation}
Thanks to \eqref{strong convergence cut M}, we have
$$\lim\sup_{\varepsilon,\mu\to0}\|\rho_{N}|\u_{N}|^2\u_{N}-\rho|\u|^2\u\|_{L^1(0,T;L^1(\O))}\leq \frac{C}{M}$$
for fixed $C>0$ and  all $M>0.$ Letting $M\to\infty$, we have
$$\rho_{N}|\u_{N}|^2\u_{N}\to\rho|\u|^2\u\quad\text{ strongly in } L^1(0,T;L^1(\O)).$$

\end{proof}

By \eqref{estimate presure N} and $\rho_{N}^{\gamma}$ converges almost everywhere to $\rho^{\gamma},$ we have $$\rho_{N}^{\gamma}\to\rho^{\gamma}\;\;\text{strongly in } L^{1}((0,T)\times\O).$$

Meanwhile, we have to mention  the following Sobolev inequality, see \cite{BD2006,Z},
$$\|\rho^{-1}\|_{L^{\infty}(\O)}\leq C(1+\|\rho\|_{H^{k+2}(\O)})^2(1+\|\rho^{-1}\|_{L^3(\O)})^3,$$
for $k\geq \frac{3}{2}.$
Thus the estimates on density from \eqref{L1}-\eqref{L3} enable us to use the above Sobolev inequality to have
\begin{equation}
\label{positive density}
\|\rho\|_{L^{\infty}((0,T)\times\O)}\geq C(\delta,\eta)>0,\quad\text{ a. e. in } (0,T)\times\O.
\end{equation}
This enables us to have $\rho_{N}^{-10}$ converges almost everywhere to $\rho^{-10}.$  Thanks to \eqref{estimate cold presure N}, we have
$$\rho_{N}^{-10}\to\rho^{-10}\;\;\text{strongly in } L^{1}((0,T)\times\O).$$

By the above compactness, we are ready to pass into the limits as $N\to\infty$ in the approximation system  \eqref{approximated system parabolic equation}, \eqref{integral equation}. Thus, we have shown that $(\rho,\u)$
 solves $$\rho_t+\Dv(\rho\u)=\varepsilon\D\rho\;\;\;\;\text{ pointwise in }\,  (0,T)\times\O,$$
 and for any test function $\varphi$ such that the following integral hold
 \begin{equation}
\begin{split}
\label{weak formulation after N}
&\int_{\O}\rho\u(T)\varphi\,dx-\int_{\O}\m_0\varphi\,dx+\mu\int_0^T\int_{\O}\D\u\cdot\D\varphi\,dx\,dt-\int_0^T\int_{\O}(\rho\u\otimes\u):\nabla\varphi\,dx\,dt
\\&+\int_0^T\int_{\O}2\rho\mathbb{D}\u:\nabla\varphi\,dx\,dt-\int_0^T\int_{\O}\rho^{\gamma}\nabla\varphi\,dx\,dt+\eta\int_0^T\int_{\O}\rho^{-10}\nabla\varphi\;dx\,dt
\\&+\varepsilon\int_0^T\int_{\O}\nabla\rho\cdot\nabla\u\varphi\,dx\,dt=-r_0\int_0^T\int_{\O}\u\varphi\,dx\,dt-r_1\int_0^T\int_{\O}\rho|\u|^2\u\varphi\,dx\,dt
\\&-2\kappa\int_0^T\int_{\O}\D\sqrt{\rho}\nabla\sqrt{\rho}\psi\;dx\;dt
-\kappa\int_0^T\int_{\O}\D\sqrt{\rho}\sqrt{\rho}\Dv\psi\,dx\,dt+\delta\int_0^T\int_{\O}\rho\nabla\D^{9}\rho\varphi\,dx\,dt.
\end{split}
\end{equation}

Thanks to the weak lower semicontinuity of convex functions, we can pass into the limits in the energy inequality \eqref{enery inequality for the first level}, by the strong convergence of the density and velocity, we have
the following energy inequality in the sense of distributions on $(0,T)$
\begin{equation}
\label{enery inequality with varepsilon}
\begin{split}&\sup_{t\in(0,T)}E(\rho,\u)+\mu\int_0^T\int_{\O}|\D\u|^2\,dx\,dt+\int_0^T\int_{\O}\rho|\mathbb{D}\u|^2\,dx\,dt+\varepsilon\delta\int_0^T\int_{\O}|\D^5\rho|^2\,dx\,dt
\\&+\varepsilon\int_0^T\int_{\O}|\nabla\rho^{\frac{\gamma}{2}}|^2\,dx\,dt+\varepsilon\eta\int_0^T\int_{\O}|\nabla\rho^{-5}|^2\,dx\,dt
+r_0\int_0^T\int_{\O}|\u|^2\,dx\,dt
\\&+r_1\int_0^T\int_{\O}\rho|\u|^4\,dx\,dt
+\kappa\varepsilon\int_0^T\int_{\O}\rho|\nabla^2\log\rho|^2\,dx\,dt\leq E_0,
\end{split}
\end{equation}
where $$E(\rho,\u)=\int_{\O}\left(\frac{1}{2}\rho|\u|^2+\frac{\rho^{\gamma}}{\gamma-1}+\frac{\eta}{10}\rho^{-10}+\frac{\kappa}{2}|\nabla\sqrt{\rho}|^2+\frac{\delta}{2}|\nabla\D^4\rho|^2\right)\,dx.$$

Thus, we have the following Lemma on the existence of weak solutions at this level approximation system.
\begin{Proposition}
\label{Proposition on the weak solutions after N}
There exists a weak solution $(\rho,\u)$ to the following system
\begin{equation*}
\begin{split}
&\rho_t+\Dv(\rho\u)=\varepsilon\D\rho,
\\&(\rho\u)_t+\Dv(\rho\u\otimes\u)+\nabla\rho^{\gamma}-\eta\nabla\rho^{-10}-\Dv(\rho\mathbb{D}\u)-\mu\D^2\u+\varepsilon\nabla\rho\cdot\nabla\u
\\&\quad\quad\quad\quad\quad\quad=-r_0\u-r_1\rho|\u|^2\u+\kappa\rho\nabla\left(\frac{\D\sqrt{\rho}}{\sqrt{\rho}}\right)+\delta\rho\nabla\Delta^9\rho,
\end{split}
\end{equation*}
with suitable initial data, for any $T>0$. In particular, the weak solutions $(\rho,\u)$ satisfies the energy inequality \eqref{enery inequality with varepsilon} and \eqref{positive density}.
\end{Proposition}
\section{Bresch-Desjardins Entropy and vanishing limits}
The goal of this section is to deduce the Bresch-Desjardins Entropy for the approximation system in Proposition \ref{Proposition on the weak solutions after N}, and to rely on it to pass into the limits as $\varepsilon,\,\mu,\,\eta,\,\delta$ go to zero.
By \eqref{L2} and \eqref{positive density}, we have \begin{equation}
\label{density bounded and bounded below}
\rho(t,x)\geq C(\delta,\eta)>0\quad\text{ and } \rho\in L^2(0,T;H^{10}(\O))\cap L^{\infty}(0,T;H^9(\O)).
\end{equation}
\subsection{BD entropy} Thanks to \eqref{density bounded and bounded below}, we can use $\varphi=\nabla(\log\rho)$ to test the momentum equation to derive the Bresch-Desjardins entropy.
 Thus, we have
 \begin{Lemma}
 \label{Lemma BD-1}
 \begin{equation*}
 \begin{split}
 &\frac{d}{dt}\int_{\O}\left(\frac{1}{2}\rho|\u+\frac{\nabla\rho}{\rho}|^2+\frac{\delta}{2}|\nabla^9\rho|^2+\frac{\kappa}{2}|\nabla\sqrt{\rho}|^2+\frac{\rho^{\gamma}}{\gamma-1}+\frac{\rho^{-10}}{10}\right)\,dx+\eta\int_{\O}|\nabla\rho^{-5}|^2\,dx
 \\&+\int_{\O}|\nabla\rho^{\frac{\gamma}{2}}|\,dx+\delta\varepsilon\int_{\O}|\D^5\rho|^2\,dx+2\delta\int_{\O}|\D^5\rho|^2\,dx+\frac{1}{2}\int_{\O}\rho|\nabla\u-\nabla^{T}\u|^2\,dx
 \\&+\mu\int_{\O}|\D\u|^2\,dx+\kappa\int_{\O}\rho|\nabla^2\log\rho|^2\,dx+\varepsilon\int_{\O}\frac{|\D\rho|^2}{\rho}\,dx
 \\&=\varepsilon\int_{\O}\nabla\rho\cdot\nabla\u\cdot\nabla\log\rho\,dx+\varepsilon\int_{\O}\D\rho\frac{|\nabla\log\rho|^2}{2}\,dx-\varepsilon\int_{\O}\Dv(\rho\u)\frac{1}{\rho}\D\rho\,dx
 \\&-\mu\int_{\O}\D\u\cdot\nabla\D\log\rho\,dx-r_1\int_{\O}|\u|^2\u\nabla\rho\,dx-r_0\int_{\O}\frac{\u\cdot\nabla\rho}{\rho}\,dx\\&
 =R_1+R_2+R_3+R_4+R_5+R_6.
 \end{split}
 \end{equation*}
 \end{Lemma}
 We can follow the same way as in \cite{Z} to deduce the above equality, and control terms $R_i$ for $i=1,2,3,4,$ and they approach to zero as $\varepsilon\to0$ or $\mu\to 0$. We estimate $R_5$ as follows
 \begin{equation*}
 |R_5|\leq C\int_{\O}\rho|\u|^2|\nabla\u|\,dx\leq C\int_{\O}\rho|\u|^4\,dx+\frac{1}{8}\int_{\O}\rho|\nabla\u|^2\,dx,
 \end{equation*}
 and for $R_6$ we have
 \begin{equation*}
 R_6=r_0\int_{\O}\frac{\rho_t+\rho\Dv\u-\varepsilon\D\rho}{\rho}\,dx=r_0\int_{\O}(\log\rho)_t\,dx-\varepsilon r_0\int_{\O}\frac{\D\rho}{\rho}\,dx.
 \end{equation*}
 since $\rho$ is uniformly bounded in $L^{\infty}(0,T;L^{\gamma}(\O))$, we have
 $$r_0\int_{\O}\log_{+}\rho\,dx\leq C,\;\;\text{ where } \log_{+}g=\log\max(g,1).$$
 Thus, we need to assume that $-r_0\int_{\O}\log_{-}\rho_0\,dx$ is uniformly bounded in $L^1(\O).$
 Also we can control
 $$\left|\varepsilon r_0\int_{\O}\frac{\D\rho}{\rho}\,dx\right|\leq \varepsilon\|\rho\|_{H^2(\O)}\|\rho^{-1}\|_{L^{\infty}(\O)},$$
 and it goes to zero as $\varepsilon\to 0.$

 Thus, we have the following inequality

 \begin{equation}
 \begin{split}
 \label{BD first level}
 &\int_{\O}\left(\frac{1}{2}\rho|\u+\frac{\nabla\rho}{\rho}|^2+\frac{\delta}{2}|\nabla^9\rho|^2+\frac{\kappa}{2}|\nabla\sqrt{\rho}|^2+\frac{\rho^{\gamma}}{\gamma-1}+\frac{\rho^{-10}}{10}-r_0\log\rho\right)\,dx+\eta\int_0^T\int_{\O}|\nabla\rho^{-5}|^2\,dx\,dt
 \\&+\int_0^T\int_{\O}|\nabla\rho^{\frac{\gamma}{2}}|\,dx\,dt+\delta\varepsilon\int_0^T\int_{\O}|\D^5\rho|^2\,dx\,dt+2\delta\int_0^T\int_{\O}|\D^5\rho|^2\,dx\,dt
 \\&+\frac{1}{2}\int_0^T\int_{\O}\rho|\nabla\u-\nabla^{T}\u|^2\,dx\,dt
 +\mu\int_0^T\int_{\O}|\D\u|^2\,dx+\kappa\int_{\O}\rho|\nabla^2\log\rho|^2\,dx\,dt\\
 &\leq \sum_{i=1}^4 R_i+\varepsilon\|\rho\|_{H^2(\O)}\|\rho^{-1}\|_{L^{\infty}(\O)}+C\int_0^T\int_{\O}\rho|\u|^4\,dx\,dt+\frac{1}{8}\int_0^T\int_{\O}\rho|\nabla\u|^2\,dx\,dt+\\&
 \int_{\O}\left(\frac{1}{2}\rho_0|\u_0+\frac{\nabla\rho_0}{\rho_0}|^2+\frac{\delta}{2}|\nabla^9\rho_0|^2+\frac{\kappa}{2}|\nabla\sqrt{\rho_0}|^2+\frac{\rho_0^{\gamma}}{\gamma-1}+\frac{\rho_0^{-10}}{10}-r_0\log_{-}\rho_0\right)\,dx
 ,
 \end{split}
 \end{equation}
where
  $\int_0^T\int_{\O}\rho|\u|^4\,dx\,dt$ is bounded by the initial energy, and $\frac{1}{8}\int_0^T\int_{\O}\rho|\nabla\u|^2\,dx\,dt$
 can be controlled by $$\int_0^T\int_{\O}\rho|\nabla\u-\nabla^T\u|^2\,dx\,dt,\quad\quad\text{ and }\;\;\int_0^T\int_{\O}\rho|\mathbb{D}\u|^2\,dx\,dt.$$
In-deed, it can be controlled by
$$\int_0^T\int_{\O}\rho|\nabla\u|^2\,dx\,dt\leq \int_{\O}\left(\rho_0|\u_0|^2+\frac{\rho^{\gamma}_0}{\gamma-1}+|\nabla\sqrt{\rho_0}|^2-r_0\log_{-}\rho_0\right)\,dx+2 E_0.$$
 Thus, \eqref{BD first level} gives us
  \begin{equation}
  \label{BD controled by initial value}
 \begin{split}
 &\int_{\O}\left(\frac{1}{2}\rho|\u+\frac{\nabla\rho}{\rho}|^2+\frac{\delta}{2}|\nabla^9\rho|^2+\frac{\kappa}{2}|\nabla\sqrt{\rho}|^2+\frac{\rho^{\gamma}}{\gamma-1}+\frac{\rho^{-10}}{10}-r_0\log\rho\right)\,dx+\eta\int_0^T\int_{\O}|\nabla\rho^{-5}|^2\,dx\,dt
 \\&+\int_0^T\int_{\O}|\nabla\rho^{\frac{\gamma}{2}}|\,dx\,dt+\delta\varepsilon\int_0^T\int_{\O}|\D^5\rho|^2\,dx\,dt+2\delta\int_0^T\int_{\O}|\D^5\rho|^2\,dx\,dt
 \\&+\frac{1}{2}\int_0^T\int_{\O}\rho|\nabla\u-\nabla^{T}\u|^2\,dx\,dt
 +\mu\int_0^T\int_{\O}|\D\u|^2\,dx+\kappa\int_{\O}\rho|\nabla^2\log\rho|^2\,dx\,dt\\
 &\leq
 2\int_{\O}\left(\frac{1}{2}\rho_0|\u_0+\frac{\nabla\rho_0}{\rho_0}|^2+\frac{\delta}{2}|\nabla^9\rho_0|^2+\frac{\kappa}{2}|\nabla\sqrt{\rho_0}|^2+\frac{\rho_0^{\gamma}}{\gamma-1}+\frac{\rho_0^{-10}}{10}-r_0\log_{-}\rho_0\right)\,dx
 \\&+ \sum_{i=1}^4 R_i+\varepsilon\|\rho\|_{H^2(\O)}\|\rho^{-1}\|_{L^{\infty}(\O)}+2E_0.
 \end{split}
 \end{equation}
Thus, we infer the following estimate from the Bresch-Desjardins entropy
 $$\kappa\int_0^T\int_{\O}\rho|\nabla^2\log\rho|^2\,dx\,dt\leq C,$$
where $C$ is independent on $\varepsilon, \;\eta,\,\mu,\,\delta.$ \\
Applying Lemma \ref{Lemma of J inequality}, we  have the following uniform estimate:
\begin{equation*}
\kappa^{\frac{1}{2}}\|\sqrt{\rho}\|_{L^2(0,T;H^2(\O))}+\kappa^{\frac{1}{4}}\|\nabla\rho^{\frac{1}{4}}\|_{L^4(0,T;L^{4}(\O))}\leq C,
\end{equation*}
where the constant $C>0$ is independent on $\varepsilon, \;\eta,\,\mu,\,\delta.$
\subsection{Passing to the limits as $\varepsilon, \mu\to0$}
We use $(\rho_{\varepsilon,\mu},\u_{\varepsilon,\mu})$ to denote the solutions at this level of approximation. It is easy to find that $(\rho_{\varepsilon,\mu},\u_{\varepsilon,\mu})$ has the following uniform estimates
\begin{equation}
\label{L1-varepsilon}
\sqrt{\rho_{\varepsilon,\mu}}\u_{\varepsilon,\mu}\in L^{\infty}(0,T;L^2(\O)),\sqrt{\rho_{\varepsilon,\mu}}\mathbb{D}\u_{\varepsilon,\mu}\in L^2((0,T)\times\O),\sqrt{\mu}\D\u_{\varepsilon,\mu}\in L^2((0,T)\times\O),
\end{equation}
\begin{equation}
\label{L2-varepsilon}
\sqrt{\varepsilon}\D^5\rho_{\varepsilon,\mu}\in L^2((0,T)\times\O),\sqrt{\delta}\rho_{\varepsilon,\mu}\in L^{\infty}(0,T;H^9(\O)),\sqrt{\kappa}\sqrt{\rho_{\varepsilon,\mu}}\in L^{\infty}(0,T;H^1(\O)),
\end{equation}
\begin{equation}
\label{L3-varepsilon}
\rho_{\varepsilon,\mu}^{-1}\in L^{\infty}(0,T;L^{10}(\O)),\sqrt{\varepsilon\eta}\nabla\rho_{\varepsilon,\mu}^{-5}\in L^2((0,T)\times\O),
\end{equation}
\begin{equation}
\label{L4-varepsilon}
\u_{\varepsilon,\mu}\in L^{2}((0,T)\times\O),\rho_{\varepsilon,\mu}^{\frac{1}{4}}\u_{\varepsilon,\mu}\in L^4((0,T)\times\O).
\end{equation}
By the Bresch-Desjardins entropy, we also have the following additional estimates
 \begin{equation}
 \label{estimate from BD}
 \nabla\sqrt{\rho_{\varepsilon,\mu}}\in L^{\infty}(0,T;L^2(\O)),\quad\sqrt{\delta}\D^5\rho_{\varepsilon,\mu}\in L^2(0,T;L^2(\O)),
 \end{equation}
 and
 \begin{equation}
 \label{presure from BD}
 \nabla\rho_{\varepsilon,\mu}^{\frac{\gamma}{2}}\in L^2((0,T)\times\O),\;\;\;\;\sqrt{\eta}\nabla\rho_{\varepsilon,\mu}^{-5}\in L^2((0,T)\times\O)).
 \end{equation}
 Also we have the following uniform estimate
 \begin{equation}
\label{estimate on density J inequality}
\kappa^{\frac{1}{2}}\|\sqrt{\rho_{\varepsilon,\mu}}\|_{L^2(0,T;H^2(\O))}+\kappa^{\frac{1}{4}}\|\nabla\rho_{\varepsilon,\mu}^{\frac{1}{4}}\|_{L^4(0,T;L^{4}(\O))}\leq C,
\end{equation}
where the constant $C>0$ is independent of $\varepsilon, \;\eta,\,\mu,\,\delta.$

By Lemma \ref{Lemma BD-1}, one deduces  $$\int_0^T\int_{\O}\rho_{\varepsilon,\mu}|\nabla\u_{\varepsilon,\mu}-\nabla^T\u_{\varepsilon,\mu}|^2\,dx\,dt\leq C,$$
which together with \eqref{L1-varepsilon}, yields
\begin{equation}
\label{L5-varepsilon}
\int_0^T\int_{\O}\rho_{\varepsilon,\mu}|\nabla\u_{\varepsilon,\mu}|^2\,dx\,dt\leq C,
\end{equation}
where the constant $C>0$ is independent of $\varepsilon, \;\eta,\,\mu,\,\delta.$
Based on above estimates, we have the following estimates uniform in $\varepsilon$:
\begin{Lemma}
\label{Lemma for uniofrm estimates}
The following estimates holds:
\begin{equation}
\label{estimate for density}
\begin{split}
&\|(\sqrt{\rho_{\varepsilon,\mu}})_t\|_{L^2(0,T;L^{2}(\O))}+\|\sqrt{\rho_{\varepsilon,\mu}}\|_{L^{2}(0,T;H^2(\O))}\leq K,\\&
\|(\rho_{\varepsilon,\mu})_t\|_{L^2(0,T;L^{\frac{3}{2}}(\O))}+\|\rho_{\varepsilon,\mu}\|_{L^{\infty}(0,T;H^{9}(\O))}+\|\rho_{\varepsilon,\mu}\|_{L^2(0,T;H^{10}(\O))}\leq K,
\end{split}
\end{equation}
\begin{equation}
\label{estimate for mass}
\|(\rho_{\varepsilon,\mu}\u_{\varepsilon,\mu})_t\|_{L^2(0,T;H^{-9}(\O))}+\|\rho_{\varepsilon,\mu}\u_{\varepsilon,\mu}\|_{L^{2}((0,T)\times\O)}\leq K,
\end{equation}
\begin{equation}
\label{nabla of mass 1}
\nabla(\rho_{\varepsilon,\mu}\u_{\varepsilon,\mu})\;\;\text{ is uniformly  bounded in } L^{4}(0,T;L^{\frac{6}{5}}(\O))+L^2(0,T;L^{\frac{3}{2}}(\O)).
\end{equation}
\begin{equation}
\label{estimate presure}
\|\rho_{\varepsilon,\mu}^{\gamma}\|_{L^{\frac{5}{3}}((0,T)\times\O)}\leq K,
\end{equation}
\begin{equation}
\label{estimate cold presure}
\|\rho_{\varepsilon,\mu}^{-10}\|_{L^{\frac{5}{3}}((0,T)\times\O)}\leq K,
\end{equation}
where $K$ is independent of $\varepsilon,\,\mu$.
\end{Lemma}
\emph{Proof.} By \eqref{L1-varepsilon}-\eqref{L5-varepsilon}, following the same way as in the proof of Lemma \ref{Lemma for uniofrm estimates N}, we can prove the above estimates.
\\

Applying Aubin-Lions Lemma and Lemma \ref{Lemma for uniofrm estimates},
we conclude
\begin{equation}
\label{strong convergence of density}
\rho_{\varepsilon,\mu}\to\rho\;\;\text{ strongly in } C(0,T;H^9(\O)),\;\text{ weakly in } L^2(0,T;H^{10}(\O)),
\end{equation}
$$\sqrt{\rho_{\varepsilon,\mu}}\to\sqrt{\rho}\;\;\text{ strongly in } L^2(0,T;H^1(\O)),\;\text{ weakly in } L^2(0,T;H^{2}(\O))$$
and \begin{equation}
\label{strong convergence of mass}
\rho_{\varepsilon,\mu}\u_{\varepsilon,\mu}\to\rho\u\;\;\text{ strongly in } L^2((0,T)\times\O).
\end{equation}
We notice that $\u_{\varepsilon,\mu}\in L^2((0,T)\times\O),$ thus, $$\u_{\varepsilon,\mu}\to \u\;\;\text{ weakly in } L^2((0,T)\times\O).$$
Thus, we can pass into the limits for term $\rho_{\varepsilon,\mu}\u_{\varepsilon,\mu}\otimes\u_{\varepsilon,\mu}$ as follows
$$\rho_{\varepsilon,\mu}\u_{\varepsilon,\mu}\otimes\u_{\varepsilon,\mu}\to\rho\u\otimes\u$$
in the distribution sense.

We can show $$\rho_{\varepsilon,\mu}|\u_{\varepsilon,\mu}|^2\u_{\varepsilon,mu}\to\rho|\u|^2\u\quad\text{ strongly in } L^1(0,T;L^1(\O))$$
as the
same to Lemma \ref{Lemma of convergence on drag term}.

Here we state the following lemma on the strong convergence of $\sqrt{\rho_{n}}\u_{n}$, which will be used later again. The proof is essential same to \cite{MV}.

\begin{Lemma}
\label{Lemma of strong convergence}
If$\rho_n^{\frac{1}{4}}\u_n$ is bounded in $ L^4(0,T;L^4(\O))$, $\rho_n$ almost everywhere converges to $\rho$, $\rho_n\u_n$ almost everywhere converges to $\rho\u,$
then
$$\sqrt{\rho_n}\u_n\to\sqrt{\rho}\u\quad\text{ strongly in } L^2(0,T;L^2(\O)).$$
\end{Lemma}

\begin{proof}
Fatou's lemma yields
$$\int_{\O}\rho|\u|^4\,dx\leq \int_{\O}\lim\inf\rho_{n}|\u_{n}|^4\,dx\leq \lim\inf\int_{\O}\rho_{n}|\u_{n}|^4\,dx,$$
and hence $\rho|\u|^4$ is in $L^{1}(0,T;L^4(\O)).$

 For almost every $(t,x)$ such that when $\rho_{n}(t,x)\neq 0$, we have
$$\u_{n}=\frac{\rho_{n}\u_{n}}{\rho_{n}}\to \u.$$
For almost every $(t,x)$ such that $\rho_{n}(t,x)=0,$
then
$$\sqrt{\rho_{n}}\u_{n}\chi_{|\u_{n}|\leq M}\leq M\sqrt{\rho_{n}}=0=\sqrt{\rho}\u\chi_{|\u|\leq M}.$$
Hence,
$\sqrt{\rho_{n}}\u_{n}\chi_{|\u_{n}|\leq M}$ converges to $ \sqrt{\rho}\u\chi_{|\u|\leq M}$ almost everywhere for $(t,x).$
 Meanwhile, $
\sqrt{\rho_{n}}\u_{n}\chi_{|\u_{n}|\leq M}$ is uniformly bounded in $L^{\infty}(0,T;L^3(\O))$.

 The dominated convergence theorem gives us
\begin{equation}
\label{strong convergence sqrt density u-1}
\sqrt{\rho_{n}}\u_{n}\chi_{|\u_{n}|\leq M} \to \sqrt{\rho}\u\chi_{|\u|\leq M}\quad \text{strongly in } L^2(0,T;L^2(\O)).
\end{equation}

For any $M>0$, we have
\begin{equation}
\begin{split}&
\int_0^T\int_{\O}\left|
\sqrt{\rho_{n}}\u_{n}-\sqrt{\rho}\u\right|^2\,dx\,dt
\\&\leq
\int_0^T\int_{\O}\left|
\sqrt{\rho_{n}}\u_{n}\chi_{|\u_{n}|\leq M}-\sqrt{\rho}\u\chi_{|\u|\leq M}\right|^2\,dx\,dt
\\&+2
\int_0^T\int_{\O}|\sqrt{\rho_{n}}\u_{n}\chi_{|\u_{n}|\geq M}|^2\,dx\,dt+2
\int_0^T\int_{\O}
|\sqrt{\rho}\u\chi_{|\u|\geq M}|^2\,dx\,dt
\\&\leq\int_0^T\int_{\O}\left|
\sqrt{\rho_{n}}\u_{n}\chi_{|\u_{n}|\leq M}-\sqrt{\rho}\u\chi_{|\u|\leq M}\right|^2\,dx\,dt
\\&+\frac{2}{M^2}
\int_0^T\int_{\O}\rho_{n}|\u_{n}|^4\,dx\,dt+\frac{2}{M^2}\int_0^T\int_{\O}\rho|\u|^4\,dx\,dt.
\end{split}
\end{equation}
Thanks to \eqref{strong convergence sqrt density u-1}, we have
$$\lim\sup_{\varepsilon,\mu\to0}\|\sqrt{\rho_{n}}\u_{n}-\sqrt{\rho}\u\|_{L^2(0,T;L^2(\O))}\leq \frac{C}{M}$$
for fixed $C>0$ and  all $M>0.$ Letting $M\to\infty$, we have
$$\sqrt{\rho_{n}}\u_{n}\to\sqrt{\rho}\u\quad\text{ strongly in } L^2(0,T;L^2(\O)).$$

\end{proof}

Applying Lemma \ref{Lemma of strong convergence} with \eqref{strong convergence of density}, \eqref{strong convergence of mass} and
$$\int_0^T\int_{\O}\rho_{\varepsilon,\mu}|\u_{\varepsilon,\mu}|^4\,dx\,dt\leq C<\infty,$$
 we have
$$\sqrt{\rho_{\varepsilon,\mu}}\u_{\varepsilon,\mu}\to\sqrt{\rho}\u\text{ strongly in } L^2(0,T;L^2(\O)).$$

By \eqref{estimate presure} and $\rho_{\varepsilon,\mu}^{\gamma}$ converges almost everywhere to $\rho^{\gamma},$ we have $$\rho_{\varepsilon,\mu}^{\gamma}\to\rho^{\gamma}\;\;\text{strongly in } L^{1}((0,T)\times\O).$$
Thanks to \eqref{density bounded and bounded below}, we have $\rho_{\varepsilon,\mu}^{-10}$ converges almost everywhere to $\rho^{-10}$. Thus, with \eqref{estimate cold presure}, we obtain
$$\rho_{\varepsilon,\mu}^{-10}\to\rho^{-10}\;\;\text{strongly in } L^{1}((0,T)\times\O).$$

By previous estimates we can extract subsequences, such that $$\varepsilon\nabla\rho_{\varepsilon,\mu}\to 0\;\;\text{ strongly in } L^2((0,T)\times\O),$$
and $$\varepsilon\nabla\rho_{\varepsilon,\mu}\nabla\u_{\varepsilon,\mu}\to 0\;\;\text{ strongly in }L^1((0,T)\times\O).$$

 For the convergence of term $\mu\Delta^2\u_{\mu}$,  for any test function $\varphi\in L^2(0,T;H^2(\O)),$ we have
 \begin{equation*}
\left|\int_0^T \int_{\O}\mu\Delta^2\u_{\varepsilon,\mu}\varphi\,dx\,dt\right|\leq \sqrt{\mu}\|\sqrt{\mu}\Delta\u_{\varepsilon,\mu}\|_{L^2(0,T;L^2(\O))}\|\Delta\varphi\|_{L^2(0,T;L^2(\O))}\to 0
 \end{equation*}
 as $\mu\to 0,$ thanks to \eqref{L1-varepsilon}.

 Due to weak lower semicontinuity of convex functions we can pass into the limits in energy inequality \eqref{enery inequality with varepsilon}, we have the following Lemma.
 \begin{Lemma}
 \label{enery inequality vanishing varepsilon and mu}
\begin{equation}
\label{energy vanishing v m}
\begin{split}&\int_{\O}\left(\frac{1}{2}\rho|\u|^2+\frac{\rho^{\gamma}}{\gamma-1}+\frac{\eta}{10}\rho^{-10}+\frac{\kappa}{2}|\nabla\sqrt{\rho}|^2
+\frac{\delta}{2}|\nabla\D^4\rho|^2\right)\,dx
\\&+\int_0^T\int_{\O}\rho|\mathbb{D}\u|^2\,dx\,dt
+r_0\int_0^T\int_{\O}|\u|^2\,dx\,dt
+r_1\int_0^T\int_{\O}\rho|\u|^4\,dx\,dt
\\&\leq \int_{\O}\left(\frac{1}{2}\rho_0|\u_0|^2+\frac{\rho_0^{\gamma}}{\gamma-1}+\frac{\eta}{10}\rho_0^{-10}+\frac{\kappa}{2}|\nabla\sqrt{\rho_0}|^2+\frac{\delta}{2}|\nabla\D^4\rho_0|^2\right)\,dx,
\end{split}
\end{equation}
\end{Lemma}
Passing to the limits in \eqref{BD controled by initial value} as $\varepsilon\to0$ and $\mu\to0$, we have the following BD entropy.
\begin{Lemma}
 \label{Lemma BD}
 \begin{equation}
 \label{BD vanishing v m}
 \begin{split}
 &\int_{\O}\left(\frac{1}{2}\rho|\u+\frac{\nabla\rho}{\rho}|^2+\frac{\delta}{2}|\nabla^9\rho|^2+\frac{\kappa}{2}|\nabla\sqrt{\rho}|^2+\frac{\rho^{\gamma}}{\gamma-1}+\frac{\rho^{-10}}{10}-r_0\log\rho\right)\,dx+
 \\&\eta\int_0^T\int_{\O}|\nabla\rho^{-5}|^2\,dx\,dt
 +\int_0^T\int_{\O}|\nabla\rho^{\frac{\gamma}{2}}|\,dx\,dt+2\delta\int_0^T\int_{\O}|\D^5\rho|^2\,dx\,dt
 \\&+\frac{1}{2}\int_0^T\int_{\O}\rho|\nabla\u-\nabla^{T}\u|^2\,dx\,dt
 +\kappa\int_0^T\int_{\O}\rho|\nabla^2\log\rho|^2\,dx\,dt
 \\&\leq 2\int_{\O}\left(\frac{1}{2}\rho_0|\u_0+\frac{\nabla\rho_0}{\rho_0}|^2+\frac{\delta}{2}|\nabla^9\rho_0|^2+\frac{\kappa}{2}|\nabla\sqrt{\rho_0}|^2+\frac{\rho_0^{\gamma}}{\gamma-1}+\frac{\rho_0^{-10}}{10}-r_0\log_{-}\rho_0\right)\,dx
 \\&+2E_0.\end{split}
 \end{equation}
 \end{Lemma}

Thus, letting $\varepsilon\to0$ and $\mu\to0$, we have shown that the following existence on the approximation system.
\begin{Proposition}
\label{Lemma on the weak solutions after varepsilon and mu}
There exists the weak solutions $(\rho,\u)$ to the following system
\begin{equation*}
\begin{split}
&\rho_t+\Dv(\rho\u)=0,
\\&(\rho\u)_t+\Dv(\rho\u\otimes\u)+\nabla\rho^{\gamma}-\eta\nabla\rho^{-10}-\Dv(\rho\mathbb{D}\u)
\\&\quad\quad\quad\quad\quad\quad=-r_0\u-r_1\rho|\u|^2\u+\kappa\rho\nabla\left(\frac{\D\sqrt{\rho}}{\sqrt{\rho}}\right)+\delta\rho\nabla\Delta^9\rho,
\end{split}
\end{equation*}
with suitable initial data, for any $T>0$. In particular, the weak solutions $(\rho,\u)$ satisfies the BD entropy \eqref{BD vanishing v m} and the energy inequality \eqref{energy vanishing v m}.
\end{Proposition}
\subsection{Pass to limits as $\eta,\delta\to0$}

At this level, the weak solutions $(\rho,\u)$  satisfies the BD entropy \eqref{BD vanishing v m} and the energy inequality \eqref{energy vanishing v m}, thus we have the following regularities:
\begin{equation}
\label{estimate from energy 1}
\sqrt{\rho}\u\in L^{\infty}(0,T;L^2(\O)),\sqrt{\rho}\mathbb{D}\u\in L^2((0,T)\times\O),
\end{equation}
\begin{equation}
\label{estimates on density 1}
\sqrt{\delta}\rho\in L^{\infty}(0,T;H^9(\O)),\sqrt{\kappa}\sqrt{\rho}\in L^{\infty}(0,T;H^1(\O)),
\end{equation}
\begin{equation}
\label{cold preasure}
\eta^{1/10}\rho^{-1}\in L^{\infty}(0,T;L^{10}(\O)),\sqrt{\eta}\nabla\rho^{-5}\in L^2((0,T)\times\O),
\end{equation}
\begin{equation*}
\u\in L^{2}(0,T;L^{2}(\O)),\rho^{\frac{1}{4}}\u\in L^4((0,T)\times\O),
\end{equation*}
\begin{equation}
\label{estimates on density}
 \nabla\sqrt{\rho}\in L^{\infty}(0,T;L^2(\O)),\quad\sqrt{\delta}\D^5\rho\in L^2(0,T;L^2(\O)).
 \end{equation}
 In particular, we have \begin{equation*}
 \kappa\int_0^T\int_{\O}\rho|\nabla^2\log\rho|^2\,dx\,dt\leq C,
 \end{equation*}
 which yields
  \begin{equation}
\label{J inequality}
\kappa^{\frac{1}{2}}\|\sqrt{\rho}\|_{L^2(0,T;H^2(\O))}+\kappa^{\frac{1}{4}}\|\nabla\rho^{\frac{1}{4}}\|_{L^4(0,T;L^{4}(\O))}\leq C,
\end{equation}
where the constant $C>0$ is independent of $\eta,\,\delta.$ That is, this inequality is still true after $\eta\to0$ and $\delta\to0.$

Thus, we have the same estimates as in Lemma \ref{Lemma for uniofrm estimates} at the levels with $\eta$ and $\delta$. Thus, we deduce the same compactness for $(\rho_{\eta},\u_{\eta})$ and $(\rho_{\delta},\u_{\delta})$.
Here, we focus on the convergence of the terms $\eta\nabla\rho^{-10}$ and $\delta\rho\nabla\D^9\rho.$ Here we pass to the limits with respect to $\eta$ first, and then with respect to $\delta$. Here we state the following two lemmas.
\begin{Lemma}
\label{Lemma on eta} For any $\rho_{\eta}$ defined as in Proposition \ref{Lemma on the weak solutions after varepsilon and mu}, we have
\begin{equation*}
\begin{split}
\eta\int_0^T\int_{\O}\rho_{\eta}^{-10}\,dx\,dt\to 0
\end{split}
\end{equation*}
as $\eta\to 0.$
\end{Lemma}
\begin{proof}
By \eqref{BD vanishing v m}, we have $$\int_{\O}(\ln(\frac{1}{\rho_{\eta}}))_{+}\,dx\,\leq C(r_0)<\infty.$$
We notice that $$y\in\R^{+}\to \ln(\frac{1}{y})_{+}$$
 is a convex continuous function.
 Moreover, Fatou's Lemma yields
 \begin{equation*}
 \begin{split}
 \int_{\O}(\ln(\frac{1}{\rho}))_{+}\,dx&\leq \int_{\O}\lim\inf (\ln(\frac{1}{\rho_{\eta}}))_{+}\,dx
 \\&\leq \lim\inf_{\eta\to 0}\int_{\O}(\ln(\frac{1}{\rho_{\eta}}))_{+}\,dx,
 \end{split}
 \end{equation*}
 and hence $(\ln(\frac{1}{\rho}))_{+}$ is in $L^{\infty}(0,T;L^1(\O)).$ It allows us to conclude that
 \begin{equation}
 \label{measure 0}
 \left|\{x:|\rho(t,x)=0\}\right|=0\quad\text{ for almost every } t,
 \end{equation}
 where $|A|$ denotes the measure of set A.

 By $(\rho_{\eta})_{t}=-\nabla\rho_{\eta}\u_{\eta}-\rho_{\eta}\Dv\u_{\eta},$ and thanks to \eqref{estimate from energy 1}-\eqref{J inequality}, we have
 $$(\rho_{\eta})_t\in L^2(0,T;L^3(\O))+L^2((0,T;L^2(\O)).$$ This, together with \eqref{estimates on density 1}, up to a subsequence and the Aubin-Lions Lemma gives us that $\rho_{\eta}$ converges to $\rho$ in $L^2(0,T;L^1(\O)),$ and hence $\rho_{\eta}\to\rho\text{ a.e.} .$\\
 Thanks to \eqref{measure 0}, we deduce
 \begin{equation}
 \label{cold preasure almost everywhere convergence}
 \eta\rho_{\eta}^{-10}\to 0\;\;\text{ a.e.}
 \end{equation}
 By \eqref{cold preasure} and Poincar\'{e}'s inequality, we have a uniform bound, with respect to $\eta$, of
 \begin{equation*}
 \eta\rho_{\eta}^{-10}\in L^{\infty}(0,T;L^1(\O))\cap L^{1}(0,T;L^3(\O)).
 \end{equation*}
 The $L^p-L^q$ interpolation inequality gives
 $$\|\eta\rho_{\eta}^{-10}\|_{L^{\frac{5}{3}}(0,T;L^{\frac{5}{3}}(\O))}\leq \|\eta\rho_{\eta}^{-10}\|^{\frac{2}{5}}_{L^{\infty}(0,T;L^1(\O))} \|\eta\rho_{\eta}^{-10}\|^{\frac{3}{5}}_{ L^{1}(0,T;L^3(\O))}\leq C,$$
 and hence $\eta\rho_{\eta}^{-10}$ is uniformly bounded in $L^{\frac{5}{3}}(0,T;L^{\frac{5}{3}}(\O))$.
 This, with \eqref{cold preasure almost everywhere convergence}, yields $$\eta\rho_{\eta}^{-10}\to 0\;\;\text{ strongly  in } L^1(0,T;L^1(\O)).$$
\end{proof}
\begin{Lemma}
\label{Lemma on delta} For any $\rho_{\delta}$ defined as in Proposition \ref{Lemma on the weak solutions after varepsilon and mu}, we have, for any test function $\varphi$,
$$\delta\int_0^T\int_{\O}\rho_{\delta}\nabla\D^9\rho_{\delta}\varphi\,dx\,dt\to 0$$
as $\delta\to 0$.
\end{Lemma}
\begin{proof}
By \eqref{estimates on density 1} and  \eqref{estimates on density},
We have uniform  bounds with respect to $\delta$ of
$$\rho_{\delta}\in L^{\infty}(0,T;L^3(\O)),\;\;\sqrt{\delta}\rho_{\delta}\in L^{\infty}(0,T;H^9(\O)),\;\;\sqrt{\delta}\rho_{\delta}\in L^2(0,T;H^{10}(\O)).$$
This, with Gagliardo-Nirenberg interpolation inequality, yields
\begin{equation*}
\|\nabla^9\rho_{\delta}\|_{L^3}\leq C\|\nabla^{10}\rho_{\delta}\|_{L^2}^{\frac{18}{19}}\|\rho_{\delta}\|_{L^3}^{\frac{1}{19}}.
\end{equation*}
Thus, we have
\begin{equation*}
\int_0^T\delta\left(\int_{\O}|\nabla^9\rho_{\delta}|^3\,dx\right)^{\frac{19}{27}}\,dt\leq C\sup_{t\in(0,T)}\left(\|\rho_{\delta}\|_{L^3(\O)}\right)^{\frac{1}{9}}\int_0^T\int_{\O}\delta |\nabla^{10}\rho_{\delta}|^2\,dx\,dt,
\end{equation*}
which implies
\begin{equation}
\label{more estimate on higher density}
\delta^{\frac{9}{19}}|\nabla^9\rho_{\delta}|\in L^{\frac{19}{9}}(0,T;L^3(\O)).
\end{equation}

For the term
$$ \delta\int_0^T\int_{\O}\rho_{\delta}\nabla\D^9\rho_{\delta}\varphi\,dx\,dt=-\delta\int_0^T\int_{\O}\D^4\Dv(\rho_{\delta}\varphi)\D^5\rho_{\delta}\,dx\,dt,$$
we focus on the most difficulty term
\begin{equation*}
\begin{split}
&\left|\delta\int_0^T\int_{\O}\D^4(\nabla\rho_{\delta})\D^5\rho_{\delta}\varphi\,dx\,dt\right|
\leq C(\varphi)\int_0^T\int_{\O}\sqrt{\delta}|\nabla^{10}\rho_{\delta}|\delta^{\frac{9}{19}}|\nabla^9\rho_{\delta}|\delta^{\frac{1}{38}}\,dx\,dt
\\& \leq C(\varphi)\delta^{\frac{1}{38}}\|\sqrt{\delta}\nabla^{10}\rho_{\delta}\|_{L^2(0,T;L^2(\O))}\|\delta^{\frac{9}{19}}\nabla^9\rho_{\delta}\|_{L^{\frac{19}{9}}(0,T;L^3(\O))}
\\&\to 0
\end{split}
\end{equation*}
as $\delta\to0$, where we used \eqref{more estimate on higher density}.

We can apply the same arguments to handle the other terms from $$\delta\int_0^T\int_{\O}\D^4\Dv(\rho_{\delta}\varphi)\D^5\rho_{\delta}\,dx\,dt.$$
Thus we have $$\delta\int_0^T\int_{\O}\rho_{\delta}\nabla\D^9\rho_{\delta}\varphi\,dx\,dt\to 0$$
as $\delta\to 0$.
\end{proof}

Here we have to remark that \eqref{J inequality} is still true even after vanishing $\eta$ and $\delta$.
Thus, letting $\eta\to0$ and $\delta\to 0$, we have shown that $(\rho,\u)$ solves \eqref{goal system}.

Meanwhile,  due to weak lower semicontinuity of convex functions, we have \eqref{energy inequality for NS} by vanishing $\eta$ and $\delta$ in energy inequality \eqref{energy vanishing v m} .
Similarly, we can obtain BD-entropy \eqref{BD entropy} by passing into the limits in \eqref{BD vanishing v m} as $\eta\to0$ and $\delta\to0.$
\subsection{Other Properties}
The time evolution of the integral averages
$$t\in(0,T)\longmapsto\int_{\O}(\rho\u)(t,x)\cdot\psi(x)\,dx$$
is defined by
\begin{equation}
\begin{split}
\label{integral for mass}
&\frac{d}{dt}\int_{\O}(\rho\u)(t,x)\cdot\psi(x)\,dx=
\int_{\O}\rho\u\otimes\u:\nabla \psi\,dx+\int_{\O}\rho^{\gamma}\Dv\psi\,dx+\int_{\O}\rho\mathbb{D}\u\nabla\psi\,dx\\&+
r_0\int_{\O}\u\psi\,dx\,+r_1\int_{\O}\rho|\u|^2\u\psi\,dx+2\kappa\int_{\O}\D\sqrt{\rho}\nabla\sqrt{\rho}\psi\;dx+\kappa\int_{\O}\D\sqrt{\rho}\sqrt{\rho}\Dv\psi\,dx.
\end{split}
\end{equation}
All estimates from \eqref{energy inequality for NS} and \eqref{BD entropy} imply \eqref{integral for mass} is continuous function with respect to $t\in[0,T].$
On the other hand, we have $$\rho\u\in L^{\infty}(0,T;L^{\frac{3}{2}}(\O)\cap L^{4}(0,T;L^2(\O)),$$
and hence $$\rho\u\in C([0,T];L_{weak}^{\frac{3}{2}}(\O)).$$

We notice $$(\sqrt{\rho})_t=-\frac{1}{2}\sqrt{\rho}\Dv\u-\nabla\sqrt{\rho}\cdot\u,$$
thus \begin{equation*}
\label{estimate of sqrt density}
\|(\sqrt{\rho})_t\|_{L^2((0,T)\times\O)}\leq C\|\sqrt{\rho}\Dv\u\|_{L^2((0,T)\times\O)}+C\|\nabla\rho^{\frac{1}{4}}\|_{L^4((0,T)\times\O)}\|\rho^{\frac{1}{4}}\u\|_{L^4((0,T)\times\O)}.
\end{equation*}
This, with $\nabla\sqrt{\rho_{\kappa}}\in L^{\infty}(0,T;L^2(\O)),$ we have
\begin{equation}
\label{strong convergence kappa level}\sqrt{\rho_{\kappa}}\to\sqrt{\rho}\text{ strongly in } L^2(0,T;L^2(\O)).
\end{equation}
Because $$(\rho\u)_t=-\Dv(\rho\u\otimes\u)-\nabla\rho^{\gamma}+\Dv(\rho\mathbb{D}\u)-r_0\u-r_1\rho|\u|^2\u+\kappa\rho\nabla\left(\frac{\D\sqrt{\rho}}{\sqrt{\rho}}\right),$$
thus we have $(\rho\u)_t$ is bounded in $L^4(0,T;W^{-1,4}(\O)).$ Meanwhile, we have
\begin{equation*}
\begin{split}
\nabla(\rho\u)=(\rho^{\frac{1}{4}}\u)\cdot\nabla\sqrt{\rho}\rho^{\frac{1}{4}}+\sqrt{\rho}\sqrt{\rho}\nabla\u,
\end{split}
\end{equation*}
which yields $\nabla(\rho\u)\in  L^4(0,T;L^{\frac{6}{5}}(\O))+L^{2}(0,T;L^{\frac{3}{2}}(\O)).$
 The Aubin-Lions Lemma gives us
\begin{equation}
\label{strong convergence of mass kappa}
\rho_{\kappa}\u_{\kappa}\to\rho\u\;\;\;\text{strongly in } L^2((0,T)\times\O).
\end{equation}

Applying Lemma \ref{Lemma of strong convergence} with \eqref{strong convergence kappa level}, \eqref{strong convergence of mass kappa}, and
$$\int_0^T\int_{\O}\rho_{\kappa}|\u_{\kappa}|^4\,dx\,dt\leq C<\infty,$$
 we have
$$\sqrt{\rho_{\kappa}}\u_{\kappa}\to\sqrt{\rho}\u\text{ strongly in } L^2(0,T;L^2(\O)).$$

\section{acknowledgement}
A. Vasseur's research was supported in part by NSF grant DMS-1209420. C. Yu's reserch was
supported in part by an AMS-Simons Travel Grant.

\end{document}